\documentclass[12pt]{article}
\usepackage{graphicx}
\usepackage{amscd, amssymb, amsmath}
\usepackage{enumerate}
\usepackage{hyperref}
\usepackage{lscape}
\newenvironment{eq}{\begin{equation}}{\end{equation}}
\newenvironment{proof}{{\bf Proof}:}{\vskip 5mm }

\newtheorem{proposition}{Proposition}[subsection]
\newtheorem{lemma}[proposition]{Lemma}
\newtheorem{definition}[proposition]{Definition}
\newtheorem{theorem}[proposition]{Theorem}

\newtheorem{remark}[proposition]{Remark}

\newtheorem{problem}[proposition]{Problem}
\newtheorem{construction}[proposition]{Construction}
\newcommand{\llabel}[1]{\label{#1}}
\newcommand{\comment}[1]{}
\newcommand{\sr}{\rightarrow}

\newcommand{\wt}{\widetilde}

{\vskip
3mm}

\newcommand{\spc}{{\,\,\,\,\,\,\,}}

\newcommand{\wtOb}{{\wt{\mathcal Ob}}}
\newcommand{\Ob}{{\mathcal Ob}}
\newcommand{\wOb}{\wt{\Ob}}

\begin{document}
\parskip = 2mm
\begin{center}
{\bf\Large The {$(\Pi,\lambda)$}-structures on the {C-systems} defined by universe categories}$^,$\footnote{For the published version of the paper see \url{http://www.tac.mta.ca/tac/volumes/32/4/32-04abs.html}}

\vspace{3mm}

{\large\bf Vladimir Voevodsky}\footnote{School of Mathematics, Institute for Advanced Study,
Princeton NJ, USA. e-mail: vladimir@ias.edu}
\vspace {3mm}

\end{center}
\begin{abstract}
We define the notion of a $(P,\wt{P})$-structure on a universe $p$ in a locally cartesian closed category category with a binary product structure and construct a $(\Pi,\lambda)$-structure on the C-systems $CC({\cal C},p)$ from a $(P,\wt{P})$-structure on $p$. 

We then define homomorphisms of C-systems with $(\Pi,\lambda)$-structures and functors of universe categories with $(P,\wt{P})$-structures and show that our construction is functorial relative to these definitions.
\end{abstract}

\tableofcontents


\subsection{Introduction}
\label{Sec.1}

The concept of a C-system in its present form was introduced in \cite{Csubsystems}. The type of the C-systems is constructively equivalent to the type of contextual categories defined by Cartmell in \cite{Cartmell0} and \cite{Cartmell1} but the definition of a C-system is slightly different from Cartmell's foundational definition.

In this paper, which extends the series started with \cite{Cfromauniverse}, \cite{fromunivwithPiI} and \cite{presheavesOb}, we continue to consider what might be the most important structure on C-systems - the structure that corresponds, for the syntactic C-systems, to the operations of dependent product,  $\lambda$-abstraction and application. The first C-system formulation of this structure was introduced by John Cartmell in \cite[pp. 3.37 and 3.41]{Cartmell0} as a part of what he called a strong M.L. structure. It was later studied by Thomas Streicher in \cite[p.71]{Streicher} who called a C-system (contextual category) together with such a structure a ``contextual category with products of families of types''. 

In \cite{fromunivwithPiI} we introduced an alternative formulation of this structure that we called a $(\Pi,\lambda)$-structure and constructed a bijection between the sets of Cartmell-Streicher structures and $(\Pi,\lambda)$-structures on any C-system $CC$. 

In this paper we consider the case of C-systems of the form $CC({\cal C},p)$ introduced in \cite{Cfromauniverse}. They are defined, in a functorial way, by a category $\cal C$ with a final object and a morphism $p:\wt{U}\sr U$  together with the choice of pullbacks of $p$ along all morphisms in $\cal C$. A morphism with such choices is called a universe in $\cal C$.  As a corollary of general functoriality we obtain a construction of isomorphisms that connect the C-systems $CC({\cal C},p)$ corresponding to different choices of pullbacks and different choices of final objects. It allows us to use the notation $CC({\cal C},p)$ that only mentions $\cal C$ and $p$. 

In \cite{presheavesOb} a number of results about presheaves on universe categories and on the C-systems $CC({\cal C},p)$ has been established. These results are of general nature and do not refer to the $(\Pi,\lambda)$-structures. However, they are highly useful for the constructions such as the one presented in this paper. 

The main result of the paper - Construction \ref{2015.03.17.constr3}, produces a $(\Pi,\lambda)$-structure on $CC({\cal C},p)$ from what we call a $(P,\wt{P})$-structure on $p$ and what is, in essence, two morphisms in $\cal C$ completing two other morphisms to a pullback. Its combination with the construction of \cite{fromunivwithPiI}, without the part that concerns the bijection, was originally stated in \cite[Proposition 2]{CMUtalk} with a sketch of a proof given in the 2009 version of \cite{NTS}.  It and the ideas that it is based on are among the most important ingredients of the construction of the univalent model of the Martin-Lof type theory in Kan simplicial sets. 

In view of Lemma \ref{2016.09.09.l1}, Construction \ref{2015.03.17.constr3} can be used not only to construct the $(\Pi,\lambda)$-structures on C-systems, but also to prove that such structures do not exist. It is possible, that a similar technique may be used with other systems of inference rules of type theory, for example, to show that for a given universe $p$ no model of a given kind of higher inductive types exists on $CC({\cal C},p)$. 

In the following section we define homomorphisms of C-systems with $(\Pi,\lambda)$-structures and functors of universe categories with $(P,\wt{P})$-structures and show, in Theorem \ref{2015.03.21.th1}, that our construction is functorial relative to these definitions.

Theorem \ref{2015.03.21.th1} is interesting also in that that its proof indirectly uses almost all results of \cite{presheavesOb}. On the other hand, modulo these results, the proof is very short and straightforward. 

The $(\Pi,\lambda)$-structures correspond to the $(\Pi,\lambda,app,\beta,\eta)$-system of inference rules. In \cite[Remark 4.4]{fromunivwithPiI} we outline the definitions of classes of structures that correspond to the similar systems but without the $\beta$- or $\eta$-rules. Such structures appear as natural variations of the $(\Pi,\lambda)$-structures. The results of the present paper admit straightforward modifications needed to construct and sometimes classify such partial $(\Pi,\lambda)$-structures on C-systems of the form $CC({\cal C},p)$.

One may wonder how the construction of this paper relates to the earlier ideas of Seely \cite{Seely1984} and their refinement by Clairambault and Dybjer \cite{ClairambaultDybjer2014}. This question requires further study. 

The methods of this paper are fully constructive and, in fact, almost entirely essentially algebraic. 

The paper is written in the same formalization-ready style as the previous ones. The main intended base for its formalization is Zermelo-Fraenkel theory. However, it can also be formalized in the existing formal systems for the univalent foundations such as the UniMath. 

Because of the importance of constructions for this paper we continue to use a special pair of names Problem-Construction for the specification of the goal of a construction and the description of the particular solution.

\comment{In the case of a Theorem-Proof pair one usually refers (by name or number) to the theorem when  using the proof of this theorem. This is acceptable in the case of theorems because the future use of their proofs is such that only the fact that there is a proof but not the particulars of the proof matter. 

In the case of a Problem-Construction pair the content of the construction often matters in the future use. Because of this we have to refer to the construction and not to the problem and we assign numbers both to Problems and to Constructions. }
  
We also continue to use the diagrammatic order of writing compositions of morphisms, i.e., for $f:X\sr Y$ and $g:Y\sr Z$ the composition of $f$ and $g$ is denoted by $f\circ g$. This rule applies to functions between sets, morphisms in categories, functors etc. 

For a functor $\Phi:{\cal C}\sr {\cal C}'$, we let $\Phi^{\circ}$ denote the functor $PreShv(C')\sr PreShv(C)$ given by the pre-composition with a functor $\Phi^{op}:{\cal C}^{op}\sr ({\cal C}')^{op}$, that is, 
$$\Phi^{\circ}(F)(X)=F(\Phi(X))$$
In the literature this functor is denoted both by $\Phi^*$ and $\Phi_*$ and we decided to use a new unambiguous notation instead. 

Acknowledgements are at the end of the paper.


\subsection{From $(P,\wt{P})$- to $(\Pi,\lambda)$-structures -- the construction}
\label{Sec.2}
In this section we describe a method of constructing $(\Pi,\lambda)$-structures on C-systems of the form $CC({\cal C},p)$ where $\cal C$ is a locally cartesian closed universe category $({\cal C},p)$ with a binary product structure. 

Let us recall the following definition from \cite{fromunivwithPiI}:
\begin{definition}
\llabel{2015.03.09.def1}
Let $CC$ be a C-system. A pre-$(\Pi,\lambda)$-structure on $CC$ is a pair of morphisms of presheaves 
$$\Pi:\Ob_2\sr \Ob_1$$
$$\lambda:\wtOb_2\sr \wtOb_1$$
such that the square
\begin{eq}
\llabel{2015.03.09.eq1}
\begin{CD}
\wtOb_2@>\lambda>> \wtOb_1\\
@V\partial VV @VV \partial V\\
\Ob_2 @>\Pi>> \Ob_1
\end{CD}
\end{eq}
commutes.  

A pre-$(\Pi,\lambda)$-structure is called a $(\Pi,\lambda)$-structure if the square (\ref{2015.03.09.eq1}) is a pullback.
\end{definition}
The functors $I_p$ used in the following definition are defined in \cite[Sec. 2.6]{presheavesOb}. 
\begin{definition}
\llabel{2015.03.29.def1}
Let $\cal C$ be a locally cartesian closed category with a binary product structure and $p:\wt{U}\sr U$ a universe in $\cal C$. A pre-$(P,\wt{P})$-structure on $p$ is a pair of morphisms 
$$\wt{P}:I_p(\wt{U}) \sr \wt{U}$$
$$P:I_p(U) \sr U$$
such that the square
\begin{eq}
\llabel{2009.prod.square}
\begin{CD}
I_p(\wt{U}) @>\wt{P}>> \wt{U}\\
@V I_p(p) VV @VV p V\\
I_p(U) @>P>> U
\end{CD}
\end{eq}
commutes.

A pre-$(P,\wt{P})$-structure is called a $(P,\wt{P})$-structure if the square (\ref{2009.prod.square}) is a pullback. 
\end{definition}
We will often say pre-$P$-structure (resp. $P$-structure) instead of pre-$(P,\wt{P})$-structure (resp. $(P,\wt{P})$-structure). 
\begin{problem}
\llabel{2015.03.17.prob0}
Let $({\cal C},p)$ be a locally cartesian closed universe category with a binary product structure. Let $(P,\wt{P})$ be a pre-$(P,\wt{P})$-structure on $p$. To construct a pre-$(\Pi,\lambda)$-structure on $CC({\cal C},p)$.
\end{problem}
\begin{construction}
\llabel{2015.03.17.constr3}\rm
Consider the diagram:
\begin{eq}
\llabel{2016.12.09.eq1}
\begin{CD}
\wt{\Ob}_2 @>\wt{\mu}_2>> int^{\circ}(Yo(I_p(\wt{U}))) @>int^{\circ}(Yo(\wt{P}))>> int^{\circ}(Yo(\wt{U})) @>\mu_1^{-1}>> \wOb_1\\
@V\partial VV @VV int^{\circ}(Yo(I_p(p))) V @VV int^{\circ}(Yo(p)) V @VV\partial V\\
\Ob_2 @>\mu_2>> int^{\circ}(Yo(I_p(U))) @>int^{\circ}(Yo(P))>> int^{\circ}(Yo(U)) @>\mu_1^{-1}>> \Ob_1
\end{CD}
\end{eq}
where $\mu_n$ and $\wt{\mu}_n$ are isomorphisms defined in \cite[Sec. 2.6]{presheavesOb}. The left hand side and the right hand side squares of this diagram commute because the squares in \cite[Problem 2.6.8]{presheavesOb} commute. The middle square commutes because the square (\ref{2009.prod.square}) commutes and both $Yo$ and $int^{\circ}$ are functors. Therefore, the outside rectangle commutes and we conclude that the pair of morphisms 
\begin{eq}
\llabel{2016.12.09.eq3}
\begin{CD}
\lambda=\wt{\mu}_2\circ int^{\circ}(Yo(\wt{P}))\circ \wt{\mu}_1^{-1}\\
\Pi=\mu_2\circ int^{\circ}(Yo(P))\circ \mu_1^{-1}
\end{CD}
\end{eq}
is a pre-$(\Pi,\lambda)$-structure on $CC({\cal C},p)$.
\end{construction}
\begin{lemma}
\llabel{2017.01.07.l4}
In the context of Construction \ref{2015.03.17.constr3}, if $(P,\wt{P})$ is a $(P,\wt{P})$-structure then the pre-$(\Pi,\lambda)$-structure constructed there is a $(\Pi,\lambda)$-structure.
\end{lemma}
\begin{proof}
We need to show that the external square of the diagram (\ref{2016.12.09.eq1}) is a pullback. 

Horizontal composition of pullbacks is a pullback. The left hand side square is a pullback because it is a commutative square with two parallel sides being isomorphisms. The right hand side square is a pullback for the same reason. 

It remains to show that the middle square is pullback. This square is obtained by applying first the functor $Yo$ and then the functor $int^{\circ}$ to the pullback square (\ref{2009.prod.square}). 

Our claim follows now from two facts:
\begin{enumerate}
\item the Yoneda functor $Yo:{\cal C}\sr PreShv({\cal C})$ takes pullbacks to pullbacks,
\item for any functor $F:{\cal C}'\sr {\cal C}$, the functor 
$$F^{\circ}:PreShv({\cal C})\sr PreShv({\cal C}')$$
of pre-composition with $F^{op}$, takes pullbacks to pullbacks.
\end{enumerate}
We assume that these two facts are known.
\end{proof}
There is an important class of cases when the function from $(P,\wt{P})$-structures on $p$ to $(\Pi,\lambda)$-structures on $CC({\cal C},p)$ defined by Construction \ref{2015.03.17.constr3} is a bijection. 
\begin{lemma}
\llabel{2016.09.09.l1}
Let $({\cal C},p)$ be a universe category such that the functor 
$$Yo\circ int^{\circ}:{\cal C}\sr PreShv(CC({\cal C},p))$$
is fully faithful. Then the function from the pre-$(P,\wt{P})$-structures on $p$ to the pre-$(\Pi,\lambda)$-structures on $CC({\cal C},p)$ defined by Construction \ref{2015.03.17.constr3} is a bijection.

Moreover, the restriction of this function to the function from $(P,\wt{P})$-structures to $(\Pi,\lambda)$-structures, which is defined in view of Lemma \ref{2017.01.07.l4},  is a bijection as well.
\end{lemma}
\begin{proof}
Let 
$$\wt{\alpha}:Mor_{PreShv(CC({\cal C},p))}(int^{\circ}(Yo(I_p(\wt{U}))),int^{\circ}(Yo(\wt{U})))\sr Mor_{\cal C}(I_p(\wt{U}),\wt{U})$$
$${\alpha}:Mor_{PreShv(CC({\cal C},p))}(int^{\circ}(Yo(I_p({U}))),int^{\circ}(Yo({U})))\sr Mor_{\cal C}(I_p({U}),{U})$$
be the inverses to $(Yo\circ int^{\circ})_{I_p(\wt{U}),\wt{U}}$ and $(Yo\circ int^{\circ})_{I_p(U),U}$ respectively. 

Given a pre-$(\Pi,\lambda)$-structure $(\Pi,\lambda)$ let
\begin{eq}
\llabel{2016.09.09.eq1}
\begin{CD}
\wt{P}=\wt{\alpha}(\wt{\mu}_2^{-1}\circ \lambda\circ\wt{\mu}_1)\\
P=\alpha(\mu_2^{-1}\circ \Pi\circ\mu_1)
\end{CD}
\end{eq}
Then $\wt{P}:I_p(\wt{U})\sr \wt{U}$ and $P:I_p(U)\sr U$. Let $S$ be the square that $\wt{P}$ and $P$ form with $I_p(p)$ and $p$. Then the square $(Yo\circ int^{\circ})(S)$ is of the form
\begin{eq}
\llabel{2017.01.07.eq7}
\begin{CD}
int^{\circ}(Yo(I_p(\wt{U}))) @>\wt{\mu}_2^{-1}\circ \lambda\circ \wt{\mu}_1>> int^{\circ}(Yo(\wt{U}))\\
@Vint^{\circ}(Yo(I_p(p))) VV @VV int^{\circ}(Yo(p)) V\\
int^{\circ}(Yo(I_p(U))) @>\mu_2^{-1}\circ \Pi\circ \mu_1>> int^{\circ}(Yo(U))
\end{CD}
\end{eq}
Since the left and right squares of (\ref{2016.12.09.eq1}) commute and their horizontal arrows are isomorphisms, the square $(Yo\circ int^{\circ})(S)$ is isomorphic to the original square formed by $\Pi$ and $\lambda$ and as a square isomorphic to a commutative square is commutative. Since $Yo\circ int^{\circ}$ is faithful, that is, injective on morphisms between a given pair of objects we conclude that $S$ is commutative, that is, $(P,\wt{P})$ defined in (\ref{2016.09.09.eq1}) is a pre-$(P,\wt{P})$-structure. 

One verifies immediately that the function from pre-$(\Pi,\lambda)$-structures to pre-$(P,\wt{P})$-structures that this construction defines is both left and right inverse to the function defined by Construction \ref{2015.03.17.constr3}. 

Assume now that we started with a $(\Pi,\lambda)$-structure. Then the square $(Yo\circ int^{\circ})(S)$ is isomorphic to a pullback and therefore is a pullback. By our assumption, the functor $Yo\circ int^{\circ}$ is fully-faithful. Fully-faithful functors reflect pullbacks, that is, if the image of a square under a fully-faithful functor is a pullback than the original square is a pullback. We conclude that both the direct and the inverse bijections map the subsets of $(P,\wt{P})$-structures and $(\Pi,\lambda)$-structures to each other. Therefore, e.g. by \cite[Lemma 5.1]{fromunivwithPiI}, the restrictions of the total bijections to these subsets are bijections as well. 

The lemma is proved.
\end{proof}
\begin{problem}
\llabel{2016.12.09.prob2}
Let $({\cal C},p)$ be a universe category. 

To construct a function from the set of $(P,\wt{P})$-structures on $p$ to the set of structures of products of families of types on $CC({\cal C},p)$. 

To show that if the functor $Yo\circ int^{\circ}$ is fully faithful than this function is a bijection.
\end{problem}
\begin{construction}\rm
\llabel{2016.12.09.constr2}
The required function is the composition of the function of Construction \ref{2015.03.17.constr3} with the construction for \cite[Problem 4.5]{fromunivwithPiI} described in that paper.
\end{construction}
\begin{remark}\rm
\llabel{2017.01.07.rem1}
One can define a mixed $(P,\wt{P})$-structure (or pre-$(P,\wt{P})$-structure) as follows:
\begin{definition}
\llabel{2009.10.27.def1}
Let $\cal C$ be an lcc category and let $p_i:\wt{U}_i\sr U_i$, $i=1,2,3$ be three morphisms in $\cal C$. A $(P,\wt{P})$-structure on $(p_1,p_2,p_3)$ is a pullback of the form
\begin{eq}
\llabel{Pisq1}
\begin{CD}
I_{p_1}(\wt{U}_2) @>\wt{P}>> \wt{U}_3\\
@VI_{p_1}(p_2)VV @VVp_3V\\
I_{p_1}(U_2) @>P>> U_3
\end{CD}
\end{eq}
\end{definition}
Then a $(P,\wt{P})$-structure on $p$ is a $(P,\wt{P})$-structure on $(p,p,p)$. This concept can be used to construct universes in C-systems that participate in impredicative  $(\Pi,\lambda)$-structures.
\end{remark}

\subsection{From $(P,\wt{P})$- to $(\Pi,\lambda)$-structures -- the functoriality}
\label{Sec.3}

Recall that in \cite[pp. 1067-68]{fromunivwithPiI} we have constructed, for any homomorphism $H:CC\sr CC'$ of C-systems, and any $n\ge 0$, natural transformations
$$H\Ob_n:\Ob_i\sr H^{\circ}(\Ob_i)$$
where for $\Gamma\in CC$ and $T\in \Ob_i(\Gamma)$ one has
$$H\Ob_n(T)=H_{Ob}(T)$$
and
$$H\wOb_n:\wOb_i\sr H^{\circ}(\wOb_i)$$
where for $\Gamma\in CC$ and $o\in \wOb_n(\Gamma)$ one has
$$H\wOb_n(o)=H_{Mor}(o)$$
\begin{definition}
\llabel{2016.09.13.def1}
Let $H:CC\sr CC'$ be a homomorphism of C-systems. Let $(\Pi,\lambda)$ and $(\Pi',\lambda')$ be pre-$(\Pi,\lambda)$-structures on $CC$ and $CC'$ respectively.

Then $H$ is called a $(\Pi,\lambda)$-homomorphism if the following two squares commute
$$
\begin{CD}
\Ob_2 @>\Pi>> \Ob_1\\
@VH\Ob_2 VV @VVH\Ob_1 V\\
H^{\circ}(\Ob_2) @>H^{\circ}(\Pi')>> H^{\circ}(\Ob_1)
\end{CD}
\spc\spc\spc\spc
\begin{CD}
\wOb_2 @>\lambda>> \wOb_1\\
@VH\wOb_2 VV @VVH\wOb_1 V\\
H^{\circ}(\wOb_2) @>H^{\circ}(\lambda')>> H^{\circ}(\wOb_1)
\end{CD}
$$

If $(\Pi,\lambda)$ and $(\Pi',\lambda')$ are $(\Pi,\lambda)$-structures then $H$ is called a $(\Pi,\lambda)$-homomorphism if it is a $(\Pi,\lambda)$-homomorphism with respect to the corresponding pre-$(\Pi,\lambda)$-structures.
\end{definition}
Unfolding the definition of $H\Ob_i$ and $H\wOb_i$ we see that $H$ is a $(\Pi,\lambda)$-homomorphism if and only if for all $\Gamma\in CC$ one has
\begin{enumerate}
\item for all $T\in \Ob_2(\Gamma)$ one has 
\begin{eq}
\llabel{2016.09.13.eq1}
H(\Pi_{\Gamma}(T))=\Pi'_{H(\Gamma)}(H(T))
\end{eq}
\item for all $o\in \wOb_2(\Gamma)$ one has
\begin{eq}
\llabel{2016.09.13.eq2}
H(\lambda_{\Gamma}(o))=\lambda'_{H(\Gamma)}(H(o))
\end{eq}
\end{enumerate}
The morphisms $\xi$ and $\wt{\xi}$ used in the following definition are defined in \cite[Sec. 3.4]{presheavesOb}. 
\begin{definition}
\llabel{2017.01.13.def1}
Let $({\cal C},p)$ and $({\cal C}',p')$ be universe categories with locally cartesian closed and binary product structures and let $(P,\wt{P})$, $(P',\wt{P}')$ be pre-$(P,\wt{P})$-structures on $p$ and $p'$ respectively. 

A universe category functor ${\bf\Phi}=(\Phi,\phi,\wt{\phi})$ is said to be a pre-$(P,\wt{P})$-functor relative to the structures $(P,\wt{P})$ and $(P',\wt{P}')$ if the squares
\begin{eq}\llabel{2015.03.23.sq1}
\begin{CD}
\Phi(I_p(U)) @>\Phi(P) >>  \Phi(U)\\
@V\xi_{{\bf\Phi},1}VV @VV \phi V\\
I_{p'}(U') @>P'>> U'
\end{CD}
\spc\spc\spc
\begin{CD}
\Phi(I_p(\wt{U})) @>\Phi(\wt{P}) >>  \Phi(\wt{U}) \\
@V\wt{\xi}_{{\bf\Phi},1}VV @VV\wt{\phi} V\\
I_{p'}(\wt{U}')@>\wt{P}' >> \wt{U}'
\end{CD}
\end{eq}
commute. 

If $(P,\wt{P})$ and $(P',\wt{P}')$ are $(P,\wt{P})$-structures then $\bf\Phi$ satisfying the above condition is called a $(P,\wt{P})$-functor.
\end{definition}
\begin{theorem}
\llabel{2015.03.21.th1}
Let $({\cal C},p)$ and $({\cal C}',p')$ be universe categories with locally cartesian closed and binary product structures. Let let $(P,\wt{P})$ and $(P',\wt{P}')$ be pre-$(P,\wt{P})$-structures on $p$ and $p'$ respectively.

If ${\bf\Phi}=(\Phi,\phi,\wt{\phi})$ is a pre-$(P,\wt{P})$-universe category functor relative to $(P,\wt{P})$ and $(P',\wt{P}')$, then the homomorphism 
$$H(\Phi,\phi,\wt{\phi}):CC({\cal C},p)\sr CC ({\cal C}',p')$$
is a homomorphism of C-systems with pre-$(\Pi,\lambda)$-structures relative to the structures obtained from $(P,\wt{P})$ and $(P',\wt{P}')$ by Construction \ref{2015.03.17.constr3}.
\end{theorem}
\begin{proof}
We have to show that for all $\Gamma\in Ob(CC({\cal C},p))$, $T\in Ob_2(\Gamma)$  and $o\in \wOb_2(\Gamma)$ the equalities (\ref{2016.09.13.eq1}) and (\ref{2016.09.13.eq2}) hold. We will prove the first equality. The proof of the second one is strictly parallel to the proof of the first.
We have
$$H(\Pi(T))=
$$$$H(\mu_1^{-1}(\mu_2(T)\circ P))=
\mu_1^{-1}(\psi(\Gamma)\circ \Phi(\mu_2(T)\circ P)\circ \xi_0)=
$$$$\mu_1^{-1}(\psi(\Gamma)\circ \Phi(\mu_2(T))\circ \Phi(P)\circ \phi)=
\mu_1^{-1}(\psi(\Gamma)\circ \Phi(\mu_2(T))\circ \xi_{1}\circ P')=
\mu_1^{-1}(\mu_2(H(T))\circ P')=
$$$$\Pi'(H(T))$$
where the first equality holds by the definition of $\Pi$, the second by \cite[Eq. 3.45]{presheavesOb}, the third by the composition axiom of functor $\Phi$ and \cite[Eq. 3.41]{presheavesOb}, the fourth by the commutativity of (\ref{2015.03.23.sq1}), the fifth  by \cite[Eq. 3.43]{presheavesOb}, and the sixth one by the definition of $\Pi'$. 
\end{proof}

\subsection{Acknowledgements}
\label{Sec.4}

I am grateful to the Department of Computer Science and Engineering of the University of Gothenburg and Chalmers University of Technology for its the hospitality during my work on the first version of the paper.  

Work on this paper was supported by NSF grant 1100938.

This material is based on research sponsored by The United States Air Force Research Laboratory under agreement number FA9550-15-1-0053. The US Government is authorized to reproduce and distribute reprints for Governmental purposes notwithstanding  any copyright notation thereon.

The views and conclusions contained herein are those of the author and should not be interpreted as necessarily representing the official policies or endorsements, either expressed or implied, of the United States Air Force Research Laboratory, the U.S. Government or Carnegie Mellon University.


\begin{thebibliography}{10}

\bibitem{Cartmell0}
John Cartmell.
\newblock Generalised algebraic theories and contextual categories.
\newblock {\em Ph.D. Thesis, Oxford University}, 1978.
\newblock \url{http://www.cs.ru.nl/~spitters/Cartmell.pdf}.

\bibitem{Cartmell1}
John Cartmell.
\newblock Generalised algebraic theories and contextual categories.
\newblock {\em Ann. Pure Appl. Logic}, 32(3):209--243, 1986.

\bibitem{ClairambaultDybjer2014}
Pierre Clairambault and Peter Dybjer.
\newblock The biequivalence of locally cartesian closed categories and
  {M}artin-{L}\"of type theories.
\newblock {\em Math. Structures Comput. Sci.}, 24(6):e240606, 54, 2014.

\bibitem{Seely1984}
R.~A.~G. Seely.
\newblock Locally {C}artesian closed categories and type theory.
\newblock {\em Math. Proc. Cambridge Philos. Soc.}, 95(1):33--48, 1984.

\bibitem{Streicher}
Thomas Streicher.
\newblock {\em Semantics of type theory}.
\newblock Progress in Theoretical Computer Science. Birkh\"auser Boston Inc.,
  Boston, MA, 1991.
\newblock Correctness, completeness and independence results, With a foreword
  by Martin Wirsing.

\bibitem{NTS}
Vladimir Voevodsky.
\newblock Notes on type systems.
\newblock 2009--2012.
\newblock \url{https://github.com/vladimirias/old_notes_on_type_systems}.

\bibitem{CMUtalk}
Vladimir Voevodsky.
\newblock The equivalence axiom and univalent models of type theory.
\newblock {\em arXiv 1402.5556}, pages 1--11, 2010.
\newblock \url{http://arxiv.org/abs/1402.5556}.

\bibitem{Cfromauniverse}
Vladimir Voevodsky.
\newblock A {C}-system defined by a universe category.
\newblock {\em Theory Appl. Categ.}, 30(37):1181--1215, 2015.
\newblock \url{http://www.tac.mta.ca/tac/volumes/30/37/30-37.pdf}.

\bibitem{fromunivwithPiI}
Vladimir Voevodsky.
\newblock Products of families of types and {$(\Pi,\lambda)$}-structures on
  {C}-systems.
\newblock {\em Theory Appl. Categ.}, 31(36):1044--1094, 2016.
\newblock \url{http://www.tac.mta.ca/tac/volumes/31/36/31-36.pdf}.

\bibitem{Csubsystems}
Vladimir Voevodsky.
\newblock Subsystems and regular quotients of {C-systems}.
\newblock In {\em Conference on Mathematics and its Applications, (Kuwait City,
  2014)}, number 658 in Contemporary Mathematics, pages 127--137, 2016.

\bibitem{presheavesOb}
Vladimir Voevodsky.
\newblock {C}-systems defined by universe categories: presheaves.
\newblock {\em Theory Appl. Categ.}, 32(3):53 -- 112, 2017.
\newblock \url{http://www.tac.mta.ca/tac/volumes/32/3/32-03.pdf}.

\end{thebibliography}

\def\cprime{$'$}

\end{document}